\documentclass[10pt, reqno]{amsart}
\usepackage{amscd, amssymb,latexsym,amsmath, amscd, amsmath, comment, hyperref}
\usepackage[all]{xy}
\usepackage{anysize}
\usepackage[sc]{mathpazo} 
\usepackage{bbm}
\usepackage{color}
\usepackage{dsfont}
\usepackage{mathtools}
\usepackage{cases}
\usepackage[numbers]{natbib}

\marginsize{2.0cm}{2.0cm}{2.0cm}{1.9cm}
\makeatletter
\newcommand{\shorteq}{\mathrel{\mkern0.2mu\mathpalette\shorteq@\relax\mkern0.2mu}}
\newcommand{\shorteq@}[2]{\scalebox{0.5}[1]{$\m@th#1=$}}

\newcommand{\longeq}[1]{\mathrel{\mathpalette\longeq@{#1}}}
\newcommand{\longeq@}[2]{%
  \begingroup
  \sbox\z@{$\m@th#1=$}%
  \ifdim#2<\wd\z@
    \resizebox{#2}{\height}{\box\z@}%
  \else
    \ifdim#2<3\wd\z@
      \hbox to #2{$\m@th#1=\hss=\hss=\hss=$}%
    \else
      \hbox to #2{$\m@th#1=\cleaders\hbox to 0.2\wd\z@{\hss$#1=$\hss}\hfil=$}%
    \fi
  \fi
  \endgroup
}
\makeatother
\newcommand{\verteq}{\rotatebox{270}{$\,\longeq{12pt}$}}
\newcommand{\equalto}[2]{\underset{\underset{\scriptstyle\overset{\mkern4mu\verteq}{#2}}{}}{#1}}

\def\F{{\mathbb F}}
\def\Fp{{\mathbb{F}_p}}
\def\Fq{{\mathbb{F}_q}}
\def\Fqs{{\mathbb{F}_{q^2}}}
\def\Fps{{\mathbb{F}_{p^2}}}
\def\Gp{{G_p}}
\def\Gq{{G_q}}
\def\Bp{{B_p}}
\def\Bq{{B_q}}

\def\Tq{{T_q}}
\def\Tqg{{T_q^\gamma}}
\def\Bqg{{B_q^\gamma}}
\def\chig{{\chi^\gamma}}
\def\Pp{{\mathbb{P}^1(\Fp)}}
\def\Pq{{\mathbb{P}^1(\Fq)}}
\def\Pps{{\mathbb{P}^1(\Fps)}}
\def\Pqs{{\mathbb{P}^1(\Fqs)}}
\def\Ogq{{O_{\Gq}}}
\def\Ogp{{O_{\Gp}}}
\def\Sgq{{\text{Stab}_{\Gq}}}
\def\Sgp{{\text{Stab}_{\Gp}}}
\def\we{{\widehat{\epsilon}}}
\def\wet{{\widehat{\eta}}}
\def\wz{{\widehat{0}}}
\def\wx{{\widehat{x}}}
\def\i{{\mathcal{I}_1}}
\def\id{{\mathcal{I}_2}}
\def\j{{\mathcal{J}}}
\def\jd{{\mathcal{J}'}}

\numberwithin{equation}{section}

\newtheorem{theorem}{Theorem}[section] 
\newtheorem{lemma}[theorem]{Lemma}
\newtheorem{proposition}[theorem]{Proposition}
\theoremstyle{remark}
\newtheorem{remark}{Remark}
\newtheorem{notation}{Notation}

\begin{document}

        \title[Restriction problem for mod \textup{p} representations]{Restriction problem for mod $p$ representations of $\text{GL}_2$ over a finite field}
	\author{Eknath Ghate and Shubhanshi Gupta}

        \address{School of Mathematics, TIFR, Mumbai - 400005, India.}	
	\email{eghate@math.tifr.res.in}
	
        \address{Department of Basic Sciences, IITRAM, Ahmedabad - 380026, India.}
	\email{shubhanshi.gupta.20pm@iitram.ac.in}

	\date{}

        \begin{abstract}
          Let $\Fq$ be the finite field with $q = p^f$ elements. We study the restriction of two classes of mod $p$ representations of $\Gq = \text{GL}_2({\Fq})$ to
  $\Gp = \text{GL}_2(\Fp)$.
  We first study the restrictions of principal series which are obtained by induction from a Borel subgroup $\Bq$. We then analyze the restrictions of inductions from an
  anisotropic torus $\Tq$ which are related to cuspidal representations. Complete decompositions are given in both cases according to the parity of $f$.
  The proofs depend on writing down explicit orbit decompositions of $\Gp \backslash \Gq / H$ where $H = \Bq$ or $\Tq$ using the fact that $\Gq/H$ is an explicit orbit in a certain projective line, along with Mackey theory.
	\end{abstract}
	\maketitle
	

\section{Introduction}
The \textit{restriction problem}, which examines how representations behave when restricted from a group to a subgroup, is a fundamental topic in representation theory.
In general, the restriction of an irreducible representation of a group to a subgroup is neither irreducible nor completely reducible. The general theory by which the
restriction of an irreducible representation of a group decomposes into irreducible representations of its subgroup is explained by \textit{branching rules}.
One is interested in computing the multiplicity with which an irreducible representation of the subgroup appears (as a subquotient) in the restriction of an irreducible
representation of the given group. More generally, one may ask how a member of a class of representations of a group decomposes under restriction into member of this class for the subgroup. 
In Kobayashi \cite{TK}, the author gives a survey of branching problems for complex representations of real reductive groups.
See also, for instance, \cite{CC}, \cite{GGP}, \cite{VP} for other aspects of the restriction problem. \\

The decomposition of a finite tensor product of representations of a group $G$ into its irreducible representations can be considered as part of a restriction problem as it can be seen as the restriction of the representation of the group $G\times G\times\ldots\times G$ (finitely many times) to the diagonal subgroup which is isomorphic to $G$. In this case, the branching is explained by what are called Clebsch-Gordan formulas.
This situation occurs naturally in the study of the restriction of irreducible mod $p$ representations of finite groups of Lie type, and in particular in the study of the restriction of irreducible mod $p$ representations (weights) of $\text{GL}_2(\Fq)$ to $\text{GL}_2(\Fp)$, where $\Fq = \F_{p^f}$ is a finite extension of degree $f \geq 1$ of the finite field $\Fp$ with $p$ elements. See, for instance, \cite{SG}
for an explicit decomposition in the case that $\Fq$ is a degree $f = 3$ extension of $\Fp$ using the Clebsch-Gordan formula from \cite{DJG}. 
\\  

In this context, it would be interesting to study the restriction problem for a slightly more complicated but natural class of representations of these groups. Let
$\Gp\coloneqq\text{GL}_2\left(\Fp\right)$ and $\Gq\coloneqq\text{GL}_2\left(\Fq\right)$. Let $\Bp$ and $\Bq$ denote the Borel subgroups of all upper triangular matrices of
$\Gp$ and $\Gq$, respectively.  A natural problem would be to
study the restriction to $\Gp$ of principal series representations of $\Gq$ obtained by inducing a mod $p$ character of $\Bq$.  Similarly, let $T_p$ and $T_q$ denote the anisotropic
torus of $\Gp$ and $\Gq$, respectively. One could also study the restrictions from $\Gq$ to $\Gp$ of representations induced from a mod $p$ character of $\Tq$. The latter representations are closely connected to the mod $p$ reductions of complex cuspidal representations of $\Gq$. \\

In this paper, we solve the restriction problem for both these classes of representations of $G_q$ completely, by decomposing the restrictions into such representations of $\Gp$ (possibly twisted by the Steinberg representation of $\Gp$). The explicit decompositions are contained in Theorems \ref{T1} and \ref{T2} which are the main results of this paper. 
The key tool, of course, is Mackey's restriction formula. However, to use it one needs to
make the orbit decomposition of $\Gp \backslash \Gq / H$ reasonably explicit where $H = \Bq$ or $\Tq$. This is done using the projective line as a crutch, noting that $\Gq/\Bq\simeq \Pq$
and $\Gq/\Tq$ is the orbit of an element in $\Pqs$ outside $\Pq$. These orbit decompositions may be of independent interest.

\section{Restriction of principal series representations}

Recall that for a mod $p$ character $\chi$ of $\Bq$, the  mod $p$ principal series representation
of $\Gq$ is the induced representation $\text{ind}_{\Bq}^{\Gq} \chi$. 
While the complete decomposition of the restriction of such a principal series representation of $\Gq$ to $\Gp$ is not immediately evident, one might expect
that the following related principal series of $\Gp$ be contained in its restriction to $\Gp$:
$$\text{ind}_{\Bp}^{\Gp}\left({\chi|}_{\Bp}\right)\hookrightarrow {\left.\text{ind}_{\Bq}^{\Gq}\chi\right|}_{\Gp}.$$
By Frobenius reciprocity, this is equivalent to checking whether ${\chi|}_{\Bp}\hookrightarrow {\left.\text{ind}_{\Bq}^{\Gq}{\chi}\right|}_{\Bp}$,
which is evidently true as $\chi\hookrightarrow {\left.\text{ind}_{\Bq}^{\Gq}{\chi}\right|}_{\Bq}$ (again by Frobenius reciprocity).
The real question then is what the other representations appearing in the decomposition of the restriction are.
In this section, we shall give a complete answer in terms of principal series representations, the Steinberg representation and cuspidal representations of $\Gp$.

\subsection{The action of $\Gp$ on $\Pq$}
There is a natural action of $\Gq$ on the projective space $\Pq$ given by
$$\left(\begin{matrix}
    a & b \\
    c & d
\end{matrix}\right)\cdot[x:y]=[ax+by:cx+dy],$$
for $\left(\begin{matrix}
    a & b \\
    c & d
\end{matrix}\right)\in\Gq$ and $[x:y]\in\Pq$. This action is transitive with $\Sgq([1:0])=\Bq$. We identify $\Pq$ with $\{[1:x]:x\in\Fq\}\cup\{[0:1]\}$ throughout the article.

\begin{notation}
    Denote $\wx\coloneqq[1:x]$ and $\widehat{\infty}\coloneqq[0:1]$.
\end{notation}

\begin{notation}
    Fix $\epsilon\in\Fqs\setminus\Fq\text{ and }\epsilon^2\in\Fq$.
    Denote the anisotropic torus of $\Gq$ by $$T_q\coloneqq\left\{\left(\begin{matrix}
    a & b \\
    b\epsilon^2 & a
\end{matrix}\right)\in\Gq:a,b \in \Fq \text{ not both zero} \right\}.$$
Analogously, for fixed $\eta\in\Fps\setminus\Fp\text{ and }\eta^2\in\Fp$, let
$$T_p\coloneqq\left\{\left(\begin{matrix}
    a & b \\
    b\eta^2 & a
\end{matrix}\right)\in\Gp: a,b \in \Fp \text{ not both zero} \right\}$$
denote the anisotropic torus of $\Gp$.
\end{notation}

\begin{lemma}\label{L1}
  Under the action of $\Gq$, we have the orbit decomposition  
  $$\Pqs =\Ogq(  \wz)\sqcup\Ogq(\we).$$ Moreover, the stabilizers of $\wz$ and $\we$ are equal to $\Bq$ and $T_q$, respectively.
\end{lemma}
\begin{proof}
  The transitivity of the $\Gq$ action on $\Pq$ gives the orbit $\Ogq(\wz)=\Pq=\{\wx:x\in\Fq\}\sqcup \{\widehat{\infty}\}$. That $\Sgq(\wz)=\Bq$ can easily be verified.
  Note that $\we\notin \Ogq(\wz)$ since $\epsilon\in\Fqs\setminus\Fq$. If we show that $\left|\Pqs\right|=\left|\Ogq(\wz)\sqcup \Ogq(\we)\right|$, then we are done.
  Let $g=\left(\begin{matrix}
    a & b \\
    c & d
      \end{matrix}\right)\in \Sgq(\we)$, that is, $g\cdot\we=\we$.
  Then for some $0\neq\lambda\in\Fqs, ~(a+b\epsilon,c+d\epsilon)=(\lambda,\lambda\epsilon) \text{ which gives }(a+b\epsilon)\epsilon=c+d\epsilon$. As $1,\epsilon$
  are linearly independent over $\Fq$, we get $a=d$ and $c=b\epsilon^2$. Thus $\Sgq(\we)=T_q$ and so the number of elements in
  $\Ogq(\we)=\frac{\left(q^2-1\right)\left(q^2-q\right)}{q^2-1}=q^2-q$ and $\left|\Ogq(\wz)\right|+\left|\Ogq(\we)\right|=q^2+1=\left|\Pqs\right|$.
\end{proof}

\begin{notation}
    For $f>2$, define the set $X=\left\{\begin{array}{ll}
    \Fq\setminus\Fp, & \text{if } 2\nmid f,  \\
        \Fq\setminus \Fps, &\text{if }2\mid f.
    \end{array}\right.$
\end{notation}

\begin{lemma}\label{L2}
    For all $x\in X$, the set $\{1,x,x^2\}$ is linearly independent over $\Fp$.
\end{lemma}
\begin{proof}
    Let $x\in\Fq$ and let $\{1,x,x^2\}$ be linearly dependent over $\Fp$. Then there exist $a_0,a_1,a_2\in\Fp$, not all zero, such that $a_0+a_1x+a_2x^2=0$. That is, $x$ satisfies a polynomial of degree at most $2$ over $\Fp$. If either $a_2=0$ or the polynomial is reducible,
    then $x\in\Fp$. If $a_2\neq 0$ and the polynomial is irreducible, then $\Fps\simeq \Fp(x)\hookrightarrow\Fq$. But $\Fps\hookrightarrow\Fq$ iff $2\mid f$. Thus, $x\in\Fp$ when $2\nmid f$ and $x\in\Fps$ when $2\mid f$.
\end{proof}

\begin{lemma}\label{L4}
    With respect to the $\Gp$ action on $\Pq$,
    $$\mathrm{Stab}_\Gp(\wx)=Z_p,\text{ for all }x\in X,$$ where $Z_p$ is the centre of $\Gp$. 
\end{lemma}
\begin{proof}
    Let $\left(\begin{matrix}
    a & b \\
    c & d
\end{matrix}\right)\in\Sgp(\wx)$, then $(a+bx,c+dx)=(\lambda,\lambda x)$, for some $0\neq \lambda\in\Fq$. So, we get $c+(d-a)x-bx^2=0$ and by Lemma \ref{L2}, $a=d$ and $b=c=0$.
\end{proof}

\begin{proposition}\label{P1}
    The decomposition of $\Pq$ into its $\Gp\text{-orbits}$ is given as
    \begin{numcases}
{\Pq=}\Ogp(\wz)\sqcup\bigsqcup\limits_{\wx\in\i}\Ogp(\wx), \hspace{1.5cm}\text{ if }2\nmid f, \label{E11}\\
\Ogp(\wz)\sqcup\Ogp(\wet)\sqcup \bigsqcup\limits_{\wx\in\id}\Ogp(\wx), \text{ if }2\mid f,\label{E22}
    \end{numcases}
where $\eta\in\Fps\setminus\Fp$ with $\eta^2\in\Fp$ and where $\i$ and $\id$ form a complete set of representatives of the $\Gp\text{-orbits}$ in $\Pq\setminus\Pp$ and $\Pq\setminus\Pps$, respectively. Also, 
$$\left| \i\right|=\frac{p^{f-1}-1}{p^2-1} \text{ and } \left| \id\right|=\frac{p\left(p^{f-2}-1\right)}{p^2-1}.$$
\end{proposition}
\begin{proof}
  Suppose $2\nmid f$. Since $\Pp\hookrightarrow\Pq$ is $\Gp\text{-stable}$ we have
  $$\Gp\backslash\Pq=\Gp\backslash\Pp \sqcup \Gp\backslash\left(\Pq\setminus\Pp\right).$$
  Note that $\Gp\backslash\Pp=\{\Ogp(\wz)\}$ as $\Gp$ acts transitively on $\Pp$. Moreover, for any $\wx\in\Pq\setminus\Pp,~ \Sgp(\wx)=Z_p$ by Lemma \ref{L4},
  and so $\left|\Ogp(\wx)\right|=p\left(p^2-1\right)$. Since the orbit of each element in $\Pq\setminus\Pp$ has the same cardinality, the number
  of $\Gp\text{-orbits}$ in $\Pq\setminus\Pp$ is $$\left|\i\right|=\frac{\left|\Pq\setminus\Pp\right|}{\left|\Ogp(\wx)\right|}=\frac{\left(p^f+1\right)-\left(p+1\right)}{p\left(p^2-1\right)}=\frac{p^{f-1}-1}{p^2-1}.$$
  Suppose $2\mid f$. Since $\Fps\hookrightarrow\Fq$ and $\Pps\hookrightarrow\Pq$ is $\Gp\text{-stable}$, we have
  $$\Gp\backslash\Pq= \Gp\backslash\Pps \sqcup \Gp\backslash\left(\Pq\setminus\Pps\right).$$
  Note that $\Gp\backslash\Pps =  \{\Ogp(\wz),\Ogp(\wet)\}$ 
  by taking $f=1$ in Lemma \ref{L1}.
  Again, for any $\wx\in\Pq\setminus\Pps,~ \Sgp(\wx)\simeq Z_p$ by Lemma \ref{L4}, and so, $\left|\Ogp(\wx)\right|=p\left(p^2-1\right)$.
  So, by the same argument as in the previous case, the number of $\Gp\text{-orbits}$ in $\Pq\setminus\Pps$ is
       $$\left|\id\right|=\frac{\left|\Pq\setminus\Pps\right|}{\left|\Ogp(\wx)\right|}=\frac{\left(p^f+1\right)-\left(p^2+1\right)}{p\left(p^2-1\right)}=\frac{p\left(p^{f-2}-1\right)}{p^2-1}. $$
\end{proof}

\subsection{Decomposition of the restriction}

Recall that, if $H$ and $K$ are two subgroups of a group $G$, not necessarily distinct, then a double coset of some fixed element $x\in G$ is the set $\{hxk: h\in H, k\in K\}$ and it is denoted by $HxK$ \cite{HM}. The set of all the double cosets $HxK$ is denoted by $H\backslash G/K$.\\
For a mod $p$ character $\chi$ of $\Bq$, let $\text{ind}_{\Bq}^{\Gq}\chi$ be a principal series representation of $\Gq$. Then by Mackey's restriction formula, we have

\begin{equation}\label{EA}
{\left.\text{ind}_{\Bq}^{\Gq}\chi\right|}_{\Gp}\simeq\bigoplus\limits_{\gamma\in \Gamma}\text{ind}_{\Bqg}^{\Gp} \chig
\end{equation}
where $\Bqg \coloneqq\gamma\Bq \gamma^{-1}\cap \Gp$ and $\chig(b)=\chi\left(\gamma^{-1}b\gamma\right)$ for all $b\in\Bqg$ and where $\Gamma$ is a set of double coset representatives of $\Gp\backslash\Gq/\Bq$.

\begin{notation}
    For $x\in\Fq$, let $g_x\coloneqq\left(\begin{matrix}
    1 & 0 \\
    x & 1
\end{matrix}\right)$ and $w\coloneqq\left(\begin{matrix}
    0 & 1 \\
    1 & 0
\end{matrix}\right)\text{ in }\Gq$.
\end{notation}

\begin{remark}\label{R2}
    A complete set of coset representatives of $\Gq/\Bq$ is given by $\{g_x:x\in\Fq\}\cup\{w\}$, and under the identification $\Gq/\Bq\simeq \Pq$, we see that $g_x\Bq\mapsto \wx\text{ and }w\Bq\mapsto \widehat{\infty}$.
\end{remark}

Under the action of $\Gp$ on $\Gq/\Bq$, we have $\Sgp\left(g\Bq\right)=g \Bq g^{-1}\cap \Gp = B_q^g$. Hence, taking $g=g_x$ and using Remark \ref{R2}, we have
    \begin{equation}\label{EA1}         
    B_q^{g_x} = \Sgp(\wx),\text{ for all } x \in \Fq. 
    \end{equation}

Now, the double cosets in $\Gp\backslash\Gq/\Bq$ are in one-to-one correspondence with the orbits in $\Gp\backslash\left(\Gq/\Bq\right)$ and hence with the orbits in $\Gp\backslash\Pq$ by Remark \ref{R2}. Also,
$$   \Gp g_x \Bq  = \text{union of left cosets in the }\Gp \text{ orbit of }g_x\Bq.$$
   Moreover, $\Gp w\Bq=\Gp g_0\Bq$. Thus, if $\Gamma$ is a set of double coset representatives of $\Gp\backslash\Gq/\Bq$, then we may assume $\Gamma\subseteq\{g_x: x\in\Fq\}$. Again, by Remark \ref{R2},
    \begin{equation}\label{EA2}
        \Gamma\xleftrightarrow{~1-1~}\Gp\backslash \Pq\text{ where }  g_x\mapsto \Ogp(\wx).
    \end{equation} 
Thus, substituting \eqref{EA1} and \eqref{EA2} in \eqref{EA}, we get
$$\displaystyle{{\left.\text{ind}_{\Bq}^{\Gq}\chi\right|}_{\Gp}\simeq\bigoplus\limits_{\Ogp(\wx)\in \Gp\backslash \Pq}\text{ind}_{\Sgp(\wx)}^\Gp \chi^{g_x}}.$$

Now, Proposition \ref{P1} gives the description of all the $\Gp\text{-orbits}$ in $\Pq$ and therefore,
$$\displaystyle{{\left.\text{ind}_{\Bq}^{\Gq}\chi\right|}_{\Gp}\simeq\left\{\begin{array}{ll}
    \text{ind}_{\Sgp(\wz)}^\Gp \chi^{g_0}~\oplus\bigoplus\limits_{\wx\in\i}\text{ind}_{\Sgp(\wx)}^\Gp \chi^{g_x}, & \text{ if }2\nmid f, \\
   
      \text{ind}_{\Sgp(\wz)}^\Gp \chi^{g_0}~\oplus\text{ind}_{\Sgp(\wet)}^\Gp\chi^{g_{\eta}}~\oplus\bigoplus\limits_{\wx\in\id}\text{ind}_{\Sgp(\wx)}^\Gp \chi^{g_x},   & \text{ if }2\mid f.
\end{array}\right.}$$
Note that, by Lemma \ref{L1}, the stabilizers of $\wz$ and $\wet$ in $\Gp$ are $\Bp$ and $T_p$ respectively and $\chi^{g_0}={\chi|}_{\Bp}$.
Also, for all $\wx$ in $\i$ or $\id$, the corresponding field element $x$ is in $X$ (see the definition of $\i$ and $\id$) and so, $\Sgp(\wx)=Z_p$, for all $\wx$ in $\i$ or $\id$ by Lemma \ref{L4}. Lastly, we note that ${\left.\chi^{g_x}\right|}_{Z_p} ={\chi|}_{Z_p}$, for all $g_x\in\Gq$ and so $\text{ind}_{\Sgp(\wx)}^\Gp \chi^{g_x}=\text{ind}_{Z_p}^\Gp \left({\chi|}_{Z_p}\right)$, for all $\wx\in\i$ or $\wx\in\id$. Combining this information, we get

\begin{equation}\label{E1}
\displaystyle{{\left.\text{ind}_{\Bq}^{\Gq}\chi\right|}_{\Gp}\simeq\left\{\begin{array}{ll}
\text{ind}_{\Bp}^{\Gp} \left({\chi|}_{\Bp}\right)\oplus\left|\i\right|\text{ind}_{Z_p}^\Gp \left({\chi|}_{Z_p}\right),  
   & \text{ if }2\nmid f, \\
      \text{ind}_{\Bp}^{\Gp}\left({\chi|}_{\Bp}\right)\oplus \text{ind}_{T_p}^{\Gp}\chi^{g_{\eta}}\oplus
      \left|\id\right|\text{ind}_{Z_p}^{\Gp} \left({\chi|}_{Z_p}\right),
 & \text{ if }2\mid f.
\end{array}\right.}
\end{equation}

\begin{notation}
    For all positive integers $f$, any character of $\Bq$ is of the form $\chi_{r,s}$ for 
$0\leq r,s<q-1$, where $$\chi_{r,s}\left(\left(\begin{matrix}
    a & b \\
    0 & d
\end{matrix}\right)\right)=a^rd^s=a^{r-s}(ad)^s,\text{ for all }\left(\begin{matrix}
    a & b \\
    0 & d
\end{matrix}\right)\in \Bq.$$
Denote $\chi_r\coloneqq \chi_{r,0}$.
\end{notation}

Since $\text{ind}_{\Bq}^{\Gq}\left(\chi_{r,s}\right)\simeq\text{ind}_{\Bq}^{\Gq}\left(\chi_{r-s}\otimes\text{det}^s\right)\simeq\text{ind}_{\Bq}^{\Gq}\chi_{r-s}\otimes\text{det}^s$, for all $r,s\in\{0,1,\ldots,q-2\}$, it is sufficient to analyze the restriction ${\text{ind}_{\Bq}^{\Gq}\chi_r|}_{\Gp}$.

\begin{notation}
    The split torus of $\Gp$ is denoted by
    $S_p\coloneqq\left\{\left(\begin{matrix}
    a & 0 \\
    0 & d
\end{matrix}\right):a,d\in\F_p^*\right\}$.
\end{notation}


\begin{notation}
   Let $f \geq 1$. 
The character $\omega_{2f} : T_q \rightarrow \F_{q^2}^*$ is defined by 
$$\omega_{2f}\left(\left(\begin{matrix}
    a & b \\
    b\epsilon^2 & a
\end{matrix}\right)\right)=a+b\epsilon,\text{ for all }\left(\begin{matrix}
    a & b \\
    b\epsilon^2 & a
\end{matrix}\right)\in T_q.$$
All the characters of $T_q$ are then precisely of the form $\omega_{2f}^r$, where $0\leq r<q^2-1$.\\
In particular, for $f=1$, the character $\omega_2$ of $T_p$ is given as
$$\omega_2\left(\left(\begin{matrix}
    a & b \\
    b\eta^2 & a
\end{matrix}\right)\right)=a+b\eta,\text{ for all }\left(\begin{matrix}
    a & b \\
    b\eta^2 & a
\end{matrix}\right)\in T_p.$$
\end{notation}

\begin{notation}
    For a complex representation $V$ of a finite group, let $\overline{V}$ denote 
    the mod $p$ reduction of an integral model of $V$.
\end{notation}

\begin{theorem}\label{T1}
Let $0\leq r<q-1$. Then for a fixed character $\chi_r$ of $\Bq$,
\begin{enumerate}
\item $$\displaystyle{{\left.\mathrm{ind}_{\Bq}^{\Gq}\chi_r\right|}_{\Gp}\simeq\left\{\begin{array}{ll}

\mathrm{ind}_{\Bp}^{\Gp}\chi_{r}\oplus \frac{p^{f-1}-1}{p^2-1}  \left( \bigoplus\limits_{i=1}^{p-1} \mathrm{ind}_{B_p}^\Gp\chi_{i,(r-i)}\otimes\overline{\mathrm{St}} \right),  
   & \text{ if }2\nmid f, \\
   
      \mathrm{ind}_{\Bp}^{\Gp}\chi_{r}\oplus \mathrm{ind}_{T_p}^{\Gp}\omega_{2}^{r}\oplus
      \frac{p\left(p^{f-2}-1\right)}{p^2-1} \left( \bigoplus \limits_{i=1}^{p-1}\mathrm{ind}_{B_p}^\Gp \chi_{i,(r-i)} \otimes\overline{\mathrm{St}} \right),
 & \text{ if }2\mid f,
\end{array}\right.}$$ where $\mathrm{St}$ denotes the (complex) Steinberg representation of $\Gp$.

\item $$\displaystyle{{\left.\mathrm{ind}_{\Bq}^{\Gq}\chi_r\right|}_{\Gp}\simeq\left\{\begin{array}{ll}

\mathrm{ind}_{\Bp}^{\Gp}\chi_{r}\oplus \frac{p^{f-1}-1}{p^2-1}
\left(\bigoplus\limits_{i=0}^{p} \mathrm{ind}_{T_p}^\Gp\omega_{2}^{r+i(p-1)}\right),  
   & \text{ if }2\nmid f,
   \\
   
      \mathrm{ind}_{\Bp}^{\Gp}\chi_{r}\oplus \mathrm{ind}_{T_p}^{\Gp}\omega_{2}^{r}
      \oplus
      \frac{p\left(p^{f-2}-1\right)}{p^2-1} \left(\bigoplus\limits_{i=0}^{p} \mathrm{ind}_{T_p}^\Gp\omega_{2}^{r+i(p-1)}\right),
 & \text{ if }2\mid f.
\end{array}\right.}$$

\end{enumerate}
\end{theorem}

\begin{proof}
 Replace $\chi$ by $\chi_r$ in the isomorphism in \eqref{E1}. Write ${\chi_r|}_{\Bp}=\chi_r$.\\
 If $2 \mid f$, then $\chi_r^{g_{\eta}}=\omega_2^r$ as characters of $T_p$. Indeed, for all $\left(\begin{matrix}
    a & b \\
    b\eta^2 & a
\end{matrix}\right)\in T_p$,
 \begin{align*}
\chi_r^{g_{\eta}}\left(\left(\begin{matrix}
    a & b \\
    b\eta^2 & a
\end{matrix}\right)\right)&=\chi_r\left((g_{\eta})^{-1}\left(\begin{matrix}
    a & b \\
    b\eta^2 & a
\end{matrix}\right)g_{\eta}\right)\\
&=\chi_r\left(\left(\begin{matrix}
    1 & 0 \\
    -\eta & 1
\end{matrix}\right)\left(\begin{matrix}
    a & b \\
    b\eta^2 & a
\end{matrix}\right)\left(\begin{matrix}
    1 & 0 \\
    \eta & 1
\end{matrix}\right)\right)\\
&=\chi_r\left(\left(\begin{matrix}
    a+b\eta & b \\
    0 & a-b\eta
\end{matrix}\right)\right)\\
&=(a+b\eta)^{r}~~(\text{since the matrix lies in $B_q$ because $\Fps \subseteq \Fq$})   \\
&=\omega_{2}^{r}\left(\left(\begin{matrix}
    a & b \\
    b\eta^2 & a
\end{matrix}\right)\right).
\end{align*}
Also, $\left|\i\right|=\frac{p^{f-1}-1}{p^2-1}\text{ and }
\left|\id\right|=\frac{p\left(p^{f-2}-1\right)}{p^2-1}$ by Proposition \ref{P1}. By the transitivity of induction, we have $$\text{ind}_{Z_p}^\Gp\left({\chi_r|}_{Z_p}\right)\simeq\text{ind}_{S_p}^{\Gp}\text{ind}_{Z_p}^{S_p}\left({\chi_r|}_{Z_p}\right)\simeq\text{ind}_{T_p}^{\Gp}\text{ind}_{Z_p}^{T_p}\left({\chi_r|}_{Z_p}\right).$$ 
We investigate each of these isomorphisms separately.

 \begin{enumerate}
 \item Note that $\text{ind}_{Z_p}^{S_p}\left({\chi_r|}_{Z_p}\right)$  is a sum of $p-1$  characters as $S_p$ is abelian of order prime to $p$. For a character $\psi$ of $S_p$, by Frobenius reciprocity, $\psi\hookrightarrow \text{ind}_{Z_p}^{S_p}\left({\chi_r|}_{Z_p}\right)$ if and only if ${\psi|}_{Z_p}={\chi_r|}_{Z_p}$.
     Now, for all $i\in\{1,2,\ldots, p-1\}$, one can easily verify that ${\chi_r|}_{Z_p}={\left.\chi_{i,(r-i)}\right|}_{Z_p}$, and $\chi_{i,(r-i)}$ are (distinct) characters of $S_p$. So we get $\text{ind}_{Z_p}^{S_p}\left({\chi_r|}_{Z_p}\right)\simeq\bigoplus\limits_{i=1}^{p-1}\chi_{i,(r-i)}$. So 
     \begin{align*}
\text{ind}_{Z_p}^\Gp\left({\chi_r|}_{Z_p}\right)&\simeq\text{ind}_{S_p}^{\Gp}\left(\bigoplus\limits_{i=1}^{p-1}\chi_{i,(r-i)}\right)  \\
     &\simeq\bigoplus\limits_{i=1}^{p-1}\text{ind}_{S_p}^{\Gp}\chi_{i,(r-i)}\\
     &\simeq\bigoplus\limits_{i=1}^{p-1}\left(\text{ind}_{\Bp}^{\Gp}\chi_{i,(r-i)}\otimes\overline{\text{St}}\right)~~(\text{by Remark 1 in \cite{GJ}}).
\end{align*}
 
 \item Similarly, $\text{ind}_{Z_p}^{T_p}\left({\chi_r|}_{Z_p}\right)$ is a sum of $p+1$ characters. Again, for any character $\psi$ of $T_p$, we have $\psi\hookrightarrow \text{ind}_{Z_p}^{T_p}\left({\chi_r|}_{Z_p}\right)$ if and only if ${\psi|}_{Z_p}={\chi_r|}_{Z_p}$.
     Here, one can easily verify that ${\chi_r|}_{Z_p}={\left.\omega_{2}^{r+i(p-1)}\right|}_{Z_p}$, for all $i\in\{0,1,\ldots, p\}$, and $\omega_{2}^{r+i(p-1)}$ are (distinct) characters of $T_p$. So we get $\text{ind}_{Z_p}^{T_p}\left({\chi_r|}_{Z_p}\right)\simeq\bigoplus\limits_{i=0}^{p}\omega_{2}^{r+i(p-1)}$. So
\begin{align*}
\text{ind}_{Z_p}^\Gp\left({\chi_r|}_{Z_p}\right)&\simeq\text{ind}_{T_p}^{\Gp}\left(\bigoplus\limits_{i=0}^{p}\omega_{2}^{r+i(p-1)}\right)  \\
     &\simeq \bigoplus\limits_{i=0}^{p}\text{ind}_{T_p}^{\Gp}\omega_{2}^{r+i(p-1)}. \qedhere
\end{align*} 
\end{enumerate}
\end{proof}

\section{Restriction in the cuspidal case}
We know that for a non self-conjugate complex character $\chi$ of $T_q$, one has 
\begin{eqnarray}\label{E2}
\text{ind}_{T_q}^{\Gq}\chi\simeq\Theta(\chi)\otimes \text{St},
\end{eqnarray}
where $\Theta(\chi)$ is the complex cuspidal representation of $\Gq$ associated to $\chi$. Taking the mod $p$ reduction of this isomorphism, we note that the mod $p$ cuspidal representation $\overline{\Theta(\chi)}$ appears as a factor in the mod $p$ reduction $\overline{\text{ind}_{T_q}^{\Gq}\chi}$. It seems tricky to study the restriction of the former representation to $\Gp$, so instead we study the restriction of the latter representation to $G_p$. In any case, this is consistent with our theme of studying the restrictions of induced
representations. \\

As in the principal series case, one might expect that an obvious induction from the anisotropic torus $T_p$ of $\Gp$ is contained in the restriction to $\Gp$ of the induction of the mod $p$ character $\chi$ from the anisotropic torus $T_q$ of $\Gq$: $$\text{ind}_{T_p}^{\Gp}\left({\chi|}_{T_p}\right)\hookrightarrow {\left.\text{ind}_{T_q}^{\Gq}\chi\right|}_{\Gp}.$$ 
By Frobenius reciprocity, this is equivalent to checking if ${\chi|}_{T_p}\hookrightarrow {\left.\text{ind}_{T_q}^{\Gq}{\chi}\right|}_{T_p}$. This follows whenever $T_p\subseteq T_q$ since ${\chi}\hookrightarrow {\left.\text{ind}_{T_q}^{\Gq}{\chi}\right|}_{T_q}$ as representations of $T_q$ (again by Frobenius reciprocity) and restriction to $T_p$ here makes sense. However, $T_p \subseteq T_q$ is only possible if $f$ is odd (and $\epsilon$ is taken to be $\eta$, see Lemma~\ref{L5}). Indeed, if $f$ is even and 
$\left( 
\begin{matrix}
    0 & 1 \\
    \eta^2 & 0
\end{matrix} 
\right)$
lies in $T_q$, then $\epsilon^2 = \eta^2$ so that $\epsilon = \pm \eta$
lies in $\Fps \subseteq \Fq$, which is a contradiction. See also Remark \ref{R6}. Nevertheless, we shall prove in the second part of Theorem \ref{T2} the somewhat surprising fact that irrespective of the parity of $f$, one always has
$\text{ind}_{T_p}^{\Gp}\omega_2^r \hookrightarrow {\left.\text{ind}_{T_q}^{\Gq}\omega_{2f}^r\right|}_{\Gp}$ even though
${\left.\omega_{2f}^r\right|}_{T_p}=\omega_2^r$ (only) when $T_p\subseteq T_q$.

In any case, it is of interest to find all the terms in the decomposition of the restriction
which we proceed to do now.

\subsection{The action of $\Gp$ on $\Pqs$}
The action of $\Gq$ on $\Pqs$ has two orbits $\Ogq(\wz)$ and $\Ogq(\we)$ (cf. Lemma \ref{L1}). Also, the orbits of  the $\Gp$ action on $\Pqs=\mathbb{P}^1(\F_{p^{2f}})$ is given by 
\eqref{E22}. 
Hence,
\begin{align}
\equalto{{\Pqs}}{}&\overset{\Gq\text{ action}}{=}\underbrace{\Ogq(\wz)}_{\Pq}\sqcup\Ogq(\we)\label{E3}\\
\mathbb{P}^1(\F_{p^{2f}})&\overset{\Gp\text{ action}}{=}\underbrace{\Ogp(\wz)\sqcup\Ogp(\wet)}_{\Pps}\sqcup\bigsqcup\limits_{\wx\in\id}\Ogp(\wx),\label{E4}
\end{align}
where $\id$ forms a complete set of representatives of the $\Gp\text{-orbits}$ in $\mathbb{P}^1(\F_{p^{2f}})\setminus\Pps$ and $\left|\id\right|=\frac{p\left(p^{2f-2}-1\right)}{p^2-1}$.

\begin{lemma}\label{L5}
    Let $q = p^f$ with $f$ odd. Then $x\in\Fps\setminus\Fp\text{ with }x^2\in\Fp\implies x\in\Fqs\setminus\Fq\text{ with }x^2\in\Fq$.
    \end{lemma}
  \begin{proof}
      Let $f$ be odd and $x\in\Fps\setminus\Fp\text{ with }x^2\in\Fp$. Then $x\in\Fqs$ and $x^2\in\Fq$. It remains to show that $x\notin\Fq$. If not, then $x \in \Fps \cap \Fq = \Fp$, a contradiction. 
      \end{proof}

\begin{notation}
    In view of Lemma~\ref{L5}, we can and do assume that $\epsilon = \eta$ when $f$ is odd.
\end{notation} 

\begin{lemma}\label{L3}
    With respect to the above notation,
    $$\Ogp(\wet)\subseteq\left\{\begin{array}{ll}
       \Ogq(\wz),  &\text{ if $2\mid f$,}  \\
        \Ogq(\we) \text{ with }\epsilon=\eta, & \text{ if $2\nmid f$.}
    \end{array}\right.$$
\end{lemma}

\begin{proof}
    Suppose $2\mid f$. Then, $\Pps\subseteq\Pq$ and so by comparing \eqref{E3} and \eqref{E4},
       $\Ogp(\wet)\subseteq \Ogq(\wz)$. Suppose $2\nmid f$. 
       Then, taking $\epsilon=\eta$, clearly $\Ogp(\wet)\subseteq\Ogq(\we)$.
\end{proof}

\begin{proposition}\label{P2}
    Let $\epsilon\in
        \Fqs\setminus\Fq$ such that $\epsilon^2\in\Fq$. Then the decomposition of $\Ogq(\we)$ into its $\Gp\text{-orbits}$ is given by
    $$\Ogq(\we)=\left\{\begin{array}{ll}
  \bigsqcup\limits_{\wx\in\j}\Ogp(\wx),&\text{ if $2\mid f$,}\\
  
\Ogp(\we)\sqcup\bigsqcup\limits_{\wx\in\jd}\Ogp(\wx),& \text{ if $2\nmid f$,}
  \end{array}\right.$$
where $\j$ and $\jd\sqcup\{\we\}$ form a complete set of representatives of $\Gp\text{-orbits}\text{ in }\Ogq(\we)=\Pqs\setminus\Pq$ when $f$ is even and odd, respectively. Also, $\left|\j\right|=\frac{p^{f-1}\left(p^f-1\right)}{p^2-1}$ and $\left|\jd\right|=\frac{\left(p^{f-1}-1\right)\left(p^f+p-1\right)}{p^2-1}$.   
\end{proposition}

\begin{proof}
    It is obvious that $\Ogp(\wz)\subseteq\Ogq(\wz)$. By Equations \eqref{E3} and \eqref{E4}, $\Ogq(\we)$ is a (disjoint) union of some $\Gp\text{-orbits}$ of $\wx\in\id$ and possibly $\Ogp(\wet)$. In fact, by Lemma \ref{L3}, $$\Ogq(\we)=\left\{\begin{array}{ll}
  \bigsqcup\limits_{\wx\in\j}\Ogp(\wx),&\text{ if $2\mid f$,}\\
  
\Ogp(\we)\sqcup\bigsqcup\limits_{\wx\in\jd}\Ogp(\wx),& \text{ if $2\nmid f$,}
  \end{array}\right.$$
  where we recall that $\eta=\epsilon$ when $f$ is odd and, $\j$ and $\jd$ are as described in the hypothesis. Their cardinalities can be computed as follows. Note that $\left|\Ogq(\we)\right|=q^2-q$, and when $f$ is odd, $\left|\Ogp(\we)\right|=p^2-p$ since 
  $\epsilon = \eta$.
  Also, $\left| \Ogp(\wx)\right|=p(p^2-1)$ for all $\wx\in\id$, as can be seen in the proof of Proposition \ref{P1} and hence for all $\wx\in\j,\jd\subseteq\id$. Thus,
  $$\left|\j\right|=\frac{\left|\Ogq(\we)\right|}{\left|\Ogp(\wx)\right|}=\frac{q^2-q}{p(p^2-1)}=\frac{p^{f-1}\left(p^f-1\right)}{p^2-1}\text{ and }$$
$$\left|\jd\right|=\frac{\left|\Ogq(\we)\setminus\Ogp(\we)\right|}{\left|\Ogp(\wx)\right|}=\frac{(q^2-q)-(p^2-p)}{p(p^2-1)}=\frac{\left(p^{f-1}-1\right)\left(p^f+p-1\right)}{p^2-1}.$$    
\end{proof}

\begin{remark}
  To summarize, the $\Gp\text{-orbits}$ of $\Pqs$ that appear in $\Ogq(\wz)$ and $\Ogq(\we)$ respectively, are given as follows:

    \begin{align*}
        \Pqs\overset{\text{Equation }\eqref{E3}}{=}&\Ogq(\wz)\sqcup\Ogq(\we)\overset{\text{Equation }\eqref{E4}}{=}\Ogp(\wz)\sqcup\Ogp(\wet)\sqcup\bigsqcup\limits_{\wx\in\id}\Ogp(\wx)\\
       \overset{\text{Proposition }\ref{P2}}{=}& \left\{\begin{array}{ll}
\underbrace{\Ogp(\wz)\sqcup\Ogp(\wet)\sqcup\bigsqcup\limits_{\wx\in\id\setminus\j}\Ogp(\wx)}_{\Ogq(\wz)}\sqcup\underbrace{\bigsqcup\limits_{\wx\in\j}\Ogp(\wx)}_{\Ogq(\we)},&\text{ if $2\mid f$,}\\  
\underbrace{\Ogp(\wz)\sqcup\bigsqcup\limits_{\wx\in\id\setminus\jd}\Ogp(\wx)}_{\Ogq(\wz)}\sqcup\underbrace{\Ogp(\we)\sqcup\bigsqcup\limits_{\wx\in\jd}\Ogp(\wx)}_{\Ogq(\we)},&\text{ if $2\nmid f$,}
       \end{array}\right.
    \end{align*}
where $\we = \wet$ if $2 \nmid f$.

\end{remark}

\subsection{Decomposition of the restriction}
For a character $\chi$ of $T_q$, by Mackey's restriction formula, we have
\begin{equation}\label{EB}    {\left.\text{ind}_{T_q}^{\Gq}\chi\right|}_{\Gp}\simeq\bigoplus\limits_{\gamma\in \Gamma'}\text{ind}_{\Tqg}^\Gp \chig,
\end{equation}
where $\Tqg=\gamma T_q \gamma^{-1}\cap \Gp$ and $\chig(t)=\chi\left(\gamma^{-1}t \gamma\right)$ for all $t\in\Tqg$ and where $\Gamma'$ is a set of double coset representatives in $\Gp\backslash \Gq/T_q$.

\begin{remark}\label{R6}
     For $\gamma= \mathrm{Id} \in \Gamma'$, we have $\Tqg=T_q\cap \Gp=\left\{\begin{array}{cc}
        Z_p, & \text{ if }2\mid f,  \\
         T_p, & \text{ if }2\nmid f.
    \end{array}\right.$ As a consequence, $T_p\subseteq T_q$ if and only if $f$ is odd.
\end{remark}

\begin{proof}
    Clearly
    $Z_p\subseteq T_q\cap\Gp$ always. Suppose $f$ is even then we show that $T_q\cap\Gp\subseteq Z_p$. Let $g=\left(\begin{matrix}
        a & b \\
        b\epsilon^2 & a
    \end{matrix}\right)\in T_q\cap\Gp$. Then  $a,b,b\epsilon^2\in\Fp$ and $\epsilon\in\Fqs\setminus\Fq$. Suppose, if possible, that  $b\neq 0$. Then $\epsilon^2\in\Fp$. This implies that $\epsilon^p=\pm\epsilon$ and consequently $\epsilon^q=\epsilon$ since $f$ is even. This is a contradiction as $\epsilon\notin\Fq$. So, $b=0$ and $g\in Z_p$, 
    as desired. In particular, if $T_p \subseteq T_q$, then $T_p \subseteq T_q \cap \Gp = Z_p$, a contradiction.\\
    Suppose $f$ is odd. If 
    $g=\left(\begin{matrix}
        a & b \\
        b\epsilon^2 & a
    \end{matrix}\right)\in T_q\cap\Gp$, then again $a,b,b\epsilon^2\in\Fp$, but 
    since $f$ is odd, $\epsilon = \eta \in \Fps \setminus \Fp$. 
    So $g$ is in $T_p$. So $T_q\cap\Gp = T_p$. 
    %
\end{proof}

\begin{notation}
    For $a,b\in\Fq$ with $b\neq 0$, define $g_{a,b}\coloneqq\left(\begin{matrix}
        1 & 0 \\
        a & b
    \end{matrix}\right)\in\Gq$.    
\end{notation}
    
  \begin{remark}\label{R7}
  Recall that $\mathrm{Stab}_{\Gq}(\we)=T_q$  
  and $\Gq/T_q\simeq\Ogq(\we)\text{ via }T_q\mapsto\we$, by Lemma \ref{L1}. Then
  a complete set of coset representatives of $\Gq/T_q$ is given by $\{g_{a,b}:a,b\in\Fq,b\neq 0\}$. Indeed,
  \begin{align*}
      \Gq/T_q\simeq\Ogq(\we)&=\{\wx:x\in\Fqs\setminus\Fq\}\\
      &=\{\widehat{a+b\epsilon}:a,b\in\Fq, b\neq 0\}\\
      &=\{g_{a,b}\cdot\we:a,b\in\Fq,b\neq 0\}
  \end{align*} 
  and since $T_q\mapsto\we$, we get $g_{a,b}T_q\mapsto g_{a,b}\cdot\we$.
\end{remark}
    
Under the action of $\Gp$ on $\Gq/T_q$, we have $\Sgp(gT_q)=gT_q g^{-1}\cap \Gp = T_q^g$. Hence by taking $g=g_{a,b}$ and using Remark \ref{R7}, we have
    \begin{equation}\label{EB1}
      T_q^{g_{a,b}} =  \Sgp(\widehat{a+b\epsilon}),
    \text{ for all } a,b \in \Fq, b\neq 0. 
    \end{equation}

Now, the double cosets in $\Gp\backslash\Gq/T_q$ are in one-to-one correspondence with the orbits in $\Gp\backslash\left(\Gq/T_q\right)$ and hence with the orbits in $\Gp\backslash\Ogq(\we)$ by Remark \ref{R7}. Also,
   $$   \Gp g_{a,b} T_q  = \text{union of left cosets in the }\Gp\text{-orbit of }g_{a,b}T_q.$$
Thus, if $\Gamma'$ is a set of double coset representatives of $\Gp\backslash\Gq/T_q$, then we may assume that $\Gamma'\subseteq\{g_{a,b}:a,b\in\Fq, b\neq 0\}$ and again by Remark \ref{R7},
    \begin{equation}\label{EB2}
        \Gamma'\xleftrightarrow{~1-1~}\Gp\backslash (\Ogq(\we))\text{ where } g_{a,b}\mapsto \Ogp(\widehat{a+b\epsilon}).
    \end{equation} 
Thus, substituting \eqref{EB1} and \eqref{EB2} in \eqref{EB}, we get
$$\displaystyle{{\left.\text{ind}_{T_q}^{\Gq}\chi\right|}_{\Gp}\simeq\bigoplus\limits_{\Ogp(\widehat{a+b\epsilon})\in\Gp\backslash (\Ogq(\we))}\text{ind}_{\Sgp(\widehat{a+b\epsilon})}^\Gp \chi^{g_{a,b}}}.$$
    
Now, Proposition \ref{P2} gives the description of all the $\Gp\text{-orbits}$ in $\Ogq(\we)$ and therefore, 
$$\displaystyle{{\left.\text{ind}_{T_q}^{\Gq}\chi\right|}_{\Gp}\simeq\left\{\begin{array}{ll}
    \bigoplus\limits_{\widehat{a+b\epsilon}\in\j }\text{ind}_{\Sgp(\widehat{a+b\epsilon})}^\Gp \chi^{g_{a,b}},  & \text{ if }2\mid f, \\
   
      \text{ind}_{\Sgp(\we)}^\Gp \chi^{g_{0,1}}~\oplus
      \bigoplus\limits_{\widehat{a+b\epsilon}\in\jd }\text{ind}_{\Sgp(\widehat{a+b\epsilon})}^\Gp \chi^{g_{a,b}}, & \text{ if }2\nmid f.
\end{array}\right.}$$
Note that for $\wx\in\j,\jd\subseteq\id$, we have $x\in\Fqs\setminus\Fps$ by the definition of $\id$, and for such $x, ~\Sgp(\wx)=Z_p$ by Lemma \ref{L4}. Thus, $\Sgp(\widehat{a+b\epsilon})=Z_p$, for all $\widehat{a+b\epsilon}\in\j$ and $\jd$. In this case, ${\left.\chi^{g_{a,b}}\right|}_{Z_p}={\chi |}_{Z_p}$. Also, for $f$ odd, $\epsilon=\eta$ and so $\Sgp(\we)=T_p$. Note that $g_{0,1}=\mathrm{Id}$ and so $\chi^{g_{0,1}}={\chi|}_{T_p}$. Combining this information, we get
\begin{equation}\label{E5}
\displaystyle{{\left.\text{ind}_{T_q}^{\Gq}\chi\right|}_{\Gp}\simeq\left\{\begin{array}{ll}

    \left| \j\right|\text{ind}_{Z_p}^\Gp \left({\chi|}_{Z_p}\right),  & \text{ if }2\mid f, \\
   
      \text{ind}_{T_p}^\Gp \left({\chi|}_{T_p}\right)~\oplus
      \left|\jd\right|\text{ind}_{Z_p}^\Gp \left({\chi|}_{Z_p}\right), & \text{ if }2\nmid f.
\end{array}\right.}
\end{equation}

\begin{theorem}\label{T2}
Let $0\leq r<q^2-1$. Then for the character $\omega_{2f}^r$ of $T_q$,
\begin{enumerate}

\item $$\displaystyle{{\left.\mathrm{ind}_{T_q}^{\Gq}\omega_{2f}^r\right|}_{\Gp}\simeq\left\{\begin{array}{ll}

    \frac{p^{f-1}\left(p^f-1\right)}{p^2-1}\left(\bigoplus\limits_{i=1}^{p-1}\mathrm{ind}_{B_p}^\Gp \chi_{i,(r-i)}\otimes\overline{\mathrm{St}}\right),  & \text{ if }2\mid f, \\   
      \mathrm{ind}_{T_p}^\Gp\omega_2^r~\oplus
      \frac{\left(p^{f-1}-1\right)\left(p^f+p-1\right)}{p^2-1}\left(\bigoplus\limits_{i=1}^{p-1}\mathrm{ind}_{B_p}^\Gp  \chi_{i,(r-i)}\otimes\overline{\mathrm{St}}\right), & \text{ if }2\nmid f.
\end{array}\right.}$$

\item $$\displaystyle{{\left.\mathrm{ind}_{T_q}^{\Gq}\omega_{2f}^r\right|}_{\Gp}\simeq\left\{\begin{array}{ll}

    \frac{p^{f-1}\left(p^f-1\right)}{p^2-1}\left(\bigoplus\limits_{i=0}^{p} \mathrm{ind}_{T_p}^\Gp\omega_{2}^{r+i(p-1)}\right),  & \text{ if }2\mid f, \\   
      \mathrm{ind}_{T_p}^\Gp\omega_2^r~\oplus
      \frac{\left(p^{f-1}-1\right)\left(p^f+p-1\right)}{p^2-1}\left(\bigoplus\limits_{i=0}^{p} \mathrm{ind}_{T_p}^\Gp\omega_{2}^{r+i(p-1)}\right), & \text{ if }2\nmid f.
\end{array}\right.}$$

\end{enumerate}
\end{theorem}

\begin{proof}
 Replace $\chi$ by $\omega_{2f}^r$ in the isomorphism in \eqref{E5}. Then, ${\left.\omega_{2f}^r\right|}_{T_p}=\omega_2^r$ as characters of $T_p$ if $f$ is odd. Also, $\left|\j\right|=\frac{p^{f-1}\left(p^f-1\right)}{p^2-1}\text{ and }
\left|\jd\right|=\frac{\left(p^{f-1}-1\right)\left(p^f+p-1\right)}{p^2-1}$ by Proposition \ref{P2}. By the transitivity of induction, we have $$\text{ind}_{Z_p}^\Gp\left({\left.\omega_{2f}^r\right|}_{Z_p}\right)\simeq\text{ind}_{S_p}^{\Gp}\text{ind}_{Z_p}^{S_p}\left({\left.\omega_{2f}^r\right|}_{Z_p}\right)\simeq\text{ind}_{T_p}^{\Gp}\text{ind}_{Z_p}^{T_p}\left({\left.\omega_{2f}^r\right|}_{Z_p}\right).$$ 
We investigate each of these isomorphisms separately. 
 \begin{enumerate}

 \item Note that $\text{ind}_{Z_p}^{S_p}\left({\left.\omega_{2f}^r\right|}_{Z_p}\right)$ is a sum of $p-1$ characters as $S_p$ is abelian of order prime to $p$. For a character $\psi$ of $S_p$, by Frobenius reciprocity, $\psi\hookrightarrow \text{ind}_{Z_p}^{S_p}\left({\left.\omega_{2f}^r\right|}_{Z_p}\right)$ if and only if ${\psi|}_{Z_p}={\left.\omega_{2f}^r\right|}_{Z_p}$.
     Now, for all $i\in\{1,2,\ldots, p-1\}$, one can easily verify that ${\left.\omega_{2f}^r\right|}_{Z_p}={\left.\chi_{i,(r-i)}\right|}_{Z_p}$, where $\chi_{i,(r-i)}$ are (distinct) characters of $S_p$. So we get $\text{ind}_{Z_p}^{S_p}\left({\left.\omega_{2f}^r\right|}_{Z_p}\right)\simeq\bigoplus\limits_{i=1}^{p-1}\chi_{i,(r-i)}$. So,
\begin{align*}
\text{ind}_{Z_p}^\Gp\left({\left.\omega_{2f}^r\right|}_{Z_p}\right)&\simeq\text{ind}_{S_p}^{\Gp}\left(\bigoplus\limits_{i=1}^{p-1}\chi_{i,(r-i)}\right)  \\
     &\simeq \bigoplus\limits_{i=1}^{p-1}\left(\text{ind}_{S_p}^{\Gp}\chi_{i,(r-i)}\right)\\
     &\simeq\bigoplus\limits_{i=1}^{p-1}\left(\text{ind}_{\Bp}^{\Gp}\chi_{i,(r-i)}\otimes\overline{\text{St}}\right)~~( \text{by Remark 1 in \cite{GJ}}).
\end{align*}

\item Similarly, $\text{ind}_{Z_p}^{T_p}\left({\left.\omega_{2f}^r\right|}_{Z_p}\right)$ is a sum of $p+1$ characters. Again, for any character $\psi$ of $T_p$, we have $\psi\hookrightarrow \text{ind}_{Z_p}^{T_p}\left({\left.\omega_{2f}^r\right|}_{Z_p}\right)$ if and only if ${\psi|}_{Z_p}={\left.\omega_{2f}^r\right|}_{Z_p}$.
     Here, one can easily verify that ${\left.\omega_{2f}^r\right|}_{Z_p}={\left.\omega_{2}^{r+i(p-1)}\right|}_{Z_p}$, for all $i\in\{0,1,\ldots, p\}$, where $\omega_{2}^{r+i(p-1)}$ are (distinct) characters of $T_p$. So we get $\text{ind}_{Z_p}^{T_p}\left({\left.\omega_{2f}^r\right|}_{Z_p}\right)\simeq\bigoplus\limits_{i=0}^{p}\omega_{2}^{r+i(p-1)}$. Thus,
\begin{align*}
\text{ind}_{Z_p}^\Gp\left({\left.\omega_{2f}^r\right|}_{Z_p}\right)&\simeq\text{ind}_{T_p}^{\Gp}\left(\bigoplus\limits_{i=0}^{p}\omega_{2}^{r+i(p-1)}\right)  \\
     &\simeq \bigoplus\limits_{i=0}^{p}\text{ind}_{T_p}^{\Gp}\omega_{2}^{r+i(p-1)}. \qedhere
\end{align*} 
\end{enumerate}
\end{proof}

\begin{remark}\label{R4}
%
Consider $\omega_{2}$ both as a character of $T_p$ taking values in a mod $p$ field and as a characteristic zero field by taking its Teichm\"uller lift. Then, for $0\leq r<p^2-1$ with $\omega_{2}^{r}$ not self-conjugate, by using the mod $p$ reduction of the isomorphism in \eqref{E2} with $\chi=\omega_2^r$, we can further substitute $\overline{\Theta\left(\omega_{2}^{r}\right)}\otimes \overline{\mathrm{St}}$ for $\mathrm{ind}_{T_p}^{\Gp}\omega_{2}^{r}$ in the statements of Theorems \ref{T1} and \ref{T2}. We also remark that the mod $p$ reductions of complex cuspidal representations of $G_p$ have been studied in, for instance, \cite{GJ}, \cite{DR} and \cite{MS}.
 \end{remark}

 \begin{remark}
     Relative versions of the results in this paper studying the restrictions of mod $p$ representations of $\mathrm{GL}_2({\F_{q'}})$ of the type considered in this article to $\mathrm{GL}_2(\Fq)$ for $q'$ a power of $q$ could easily be further derived. 
     For brevity, 
     we have decided to treat only the absolute case $q = p$.
 \end{remark}

 \vspace{5mm}
        {\noindent \bf Acknowledgements:} The second author thanks Prof. Eknath Ghate for suggesting the problem, for his invitation to visit TIFR, Mumbai and for the warm hospitality extended during her visit. The second author also thanks Prof. Gautam Borisagar for his valuable inputs and is grateful to CSIR for the PhD fellowship.

\end{document}